\documentclass[12pt,reqno]{amsart}
\usepackage{amsmath, amsfonts, amssymb, amsthm, hyperref}
\usepackage{bm}
\allowdisplaybreaks[4]
\def\bg{\bigg}
\def\({\bg(}
\def\){\bg)}
\def\t{\text}

\def\ls{\leqslant}

\def\1{{\bf 1}}

\def\pmod #1{\ ({\rm{mod}}\ #1)}

\def\<{\langle}
\def\>{\rangle}

\theoremstyle{plain}
\newtheorem{theorem}{Theorem}[section]
\newtheorem{lemma}{Lemma}
\newtheorem{corollary}{Corollary}

\theoremstyle{definition}

\theoremstyle{remark}

\makeatletter
\@namedef{subjclassname@2020}{%
	\textup{2020} Mathematics Subject Classification}
\makeatother
\vspace{4mm}

\begin{document}
	\medskip
	
	\title[A variant of some cyclotomic matrices involving trinomial coefficients]
	{A variant of some cyclotomic matrices involving trinomial coefficients}

	\author{Yu-Bo Li}
	\address {(Yu-Bo Li) School of Science, Nanjing University of Posts and Telecommunications, Nanjing 210023, People's Republic of China}
	\email{lybmath2022@163.com}
	
	\author{Ning-Liu Wei}
	\address {(Ning-Liu Wei) School of Science, Nanjing University of Posts and Telecommunications, Nanjing 210023, People's Republic of China}
	\email{weiningliu6@163.com}

	\keywords{Central Trinomial Coefficients, Finite Fields, Determinants.
		\newline \indent 2020 {\it Mathematics Subject Classification}. Primary 05A19, 11C20; Secondary 15A18, 15B57, 33B10.}
	
	\begin{abstract}
		In this paper, by using the theory of circulant matrices we study some matrices over finite fields which involve the quadratic character and trinomial coefficients. 
	\end{abstract}
	\maketitle
	
	\section{Introduction}
	\setcounter{lemma}{0}
	\setcounter{theorem}{0}
	\setcounter{equation}{0}
	\setcounter{conjecture}{0}
	\setcounter{remark}{0}
	\setcounter{corollary}{0}
	
Let $p$ be an odd prime and let $(\cdot/p)$ be the Legendre symbol. The study of the matrices involving Legendre symbols can trace back to the works of Lehmer \cite{Lehmer} and Carlitz \cite{carlitz}. For example, Carlitz \cite{carlitz} initiated the study of the matrix 
$$C_p=\left[\left(\frac{j-i}{p}\right)\right]_{1\le i,j\le p-1}.$$
Carlitz \cite[Thm. 4 (4.9)]{carlitz} proved that the characteristic polynomial of $C_p$ is 
$$f_p(t)=\left(t^2-(-1)^{\frac{p-1}{2}}p\right)^{\frac{p-3}{2}}\left(t^2-(-1)^{\frac{p-1}{2}}\right).$$
	
Later Chapman \cite{chapman,evil-det} and Vsemirnov \cite{V1,V2} investigated many variants of Carlitz's matrix $C_p$. In particular, using sophisticated matrix decompositions Vsemirnov \cite{V1,V2} confirmed Chapman's ``evil" determinant conjecture which says that 
$$\det\left[\left(\frac{j-i}{p}\right)\right]_{1\le i,j\le \frac{p+1}{2}}=\begin{cases}
	-a_p & \mbox{if}\ p\equiv 1\pmod4,\\
	1    & \mbox{if}\ p\equiv 3\pmod4.
\end{cases}
$$
The number $a_p$ is defined by the following equality
$$\varepsilon_p^{(2-(\frac{2}{p}))h_p}=a_p+b_p\sqrt{p},\ a_p,b_p\in\mathbb{Q},$$
where $\varepsilon_p>1$ and $h_p$ denote the fundamental unit and class number
of the real quadratic field $\mathbb{Q}(\sqrt{p})$. 

Recently, Sun \cite{S19} further studied some variants of Carlitz's matrix $C_p$. For example, Sun \cite[Thm. 1.2]{S19} showed that 
$$-\det\left[\left(\frac{i^2+j^2}{p}\right)\right]_{1\le i,j\le p-1}$$
is always a quadratic residue modulo $p$. Along this line, for any integers $c,d$, the arithmetic properties of the matrix 
$$\left[\left(\frac{i^2+cij+dj^2}{p}\right)\right]_{1\le i,j\le p-1}$$
were extensively studied. Readers may refer to \cite{DPS,S19,WCR,W21,WSW} for details on this topic.

On the other hand, in the same paper Sun posed a conjecture (see \cite[Remark 1.3]{S19}) which states that 
$$2\det\left[\frac{1}{i^2-ij+j^2}\right]_{1\le i,j\le p-1}$$
is a quadratic residue modulo $p$ whenever $p\equiv2\pmod3$ is an odd prime. 
This conjecture was later proved by Wu, She and Ni \cite{WSN}. 

Also, Let 
$$D_p=\left[(i^2+j^2)\left(\frac{i^2+j^2}{p}\right)\right]_{1\le i,j\le (p-1/)2}.$$
Recently, Wu, She and Wang \cite{WSW} proved a conjecture posed by Sun which states that 
$$\left(\frac{D_p}{p}\right)=\begin{cases}
	1 & \mbox{if}\ p\equiv1\pmod4,\\
	(-1)^{\frac{h(-p)-1}{2}} & \mbox{if}\ p\equiv 3\pmod4,
\end{cases}$$
where $h(-p)$ is the class number of $\mathbb{Q}(\sqrt{-p})$. 

Now let $\mathbb{F}_q=\{0,a_1,\cdots,a_{q-1}\}$ be the finite field with $q$ elements, where $q$ is an odd prime power. Also, let $\chi$ be the unique quadratic multiplicative character of $\mathbb{F}_q$, i.e., 
$$\chi(x)=\begin{cases}
	0  & \mbox{if}\ x=0,\\
	1  & \mbox{if}\ x\ \text{is a nonzero square},\\
	-1 & \mbox{otherwise}.
\end{cases}$$

Motivated by the above results, in this paper, we shall study the following matrix over $\mathbb{F}_q$:
\begin{equation}\label{Eq. def of Wq}
	S_q:=\left[(a_i^2+a_ia_j+a_j^2)\chi(a_i^2+a_ia_j+a_j^2)\right]_{1\le i,j\le q-1}.
\end{equation}

Let $n$ be a non-negative integer. The central trinomial coefficient $T_n$ is defined to be the coefficient of $x^n$ in the polynomial $(x^2+x+1)^n$. Equivalently, $T_n$ is the constant term of $(x+1+x^{-1})^n$. Now we state our main result. 

\begin{theorem}\label{Thm. A}
	Let $q$ be an odd prime power. Then there exists an element $u\in\mathbb{F}_q$ such that 
	$$\det S_q=T_{\frac{q+1}{2}}\cdot u^2.$$
\end{theorem}

As a consequence of Theorem \ref{Thm. A}, we have the following result.

\begin{corollary}\label{Corollary A}
	Let $p$ be an odd prime. Suppose $p\nmid \det S_p$. Then 
	$$\left(\frac{\det S_p}{p}\right)=\left(\frac{T_{\frac{p+1}{2}}}{p}\right).$$
\end{corollary}

Next we shall give a necessary and sufficient condition for $S_q$ to be singular. Let $n$ be a non-negative integer. Then the trinomial coefficients $\binom{n}{k}_2$ is defined by 
$$\left(x+\frac{1}{x}+1\right)^n=\sum_{k=-n}^{n}\binom{n}{k}_2x^k.$$
Clearly $\binom{n}{0}_2=T_n$. Now we state our last result.
\begin{theorem}\label{Thm. B}
	Let $\mathbb{F}_q$ be the finite field with $q>5$ and $(q,22)=1$. Then 
	$$\det S_q=\frac{121}{64}\cdot T_{\frac{q+1}{2}}\cdot\prod_{k=1}^{(p-5)/2}\binom{\frac{q+1}{2}}{k}_2^2\in\mathbb{F}_p,$$
	where $p$ is the characteristic of $\mathbb{F}_q$. 
	Also, $S_q$ is a singular matrix over $\mathbb{F}_q$ if and only if 
	$$\binom{\frac{q+1}{2}}{k}_2\equiv 0\pmod p$$
	for some $0\le k\le (q-5)/2$.
\end{theorem}

We will prove our main results in Section 2 and Section 3 respectively. 

\section{Proof of Theorem \ref{Thm. A}}

We first introduce the definition of the circulant matrices. Let $R$ be a commutative ring. Let $m$ be a positive integer and $t_0,t_1,\ldots,t_{m-1}\in R$. We define the circulant matrix $C(t_0,\ldots,t_{m-1})$ to be an $m\times m$ matrix whose ($i$-$j$)-entry is $t_{j-i}$ where the indices are cyclic module $m$. 
Wu \cite[Lemma 3.4]{W21} obtained the following result.

\begin{lemma} \label{Lem. the key lemma}
	Let $R$ be a commutative ring. Let $m$ be a positive even integer. Let $t_0,t_1,\ldots,t_{m-1}\in R$ such that
	\begin{align*}
		t_i=t_{m-i} \ \ \t{for each $1\ls i\ls m-1$.}
	\end{align*}
	Then there exists an element $u\in R$ such that 
	$$
	\det C(t_0,\ldots,t_{m-1})=\left(\sum_{i=0}^{m-1}t_i\right)\left(\sum_{i=0}^{m-1}(-1)^it_i\right)\cdot u^2.
	$$
\end{lemma}

We also need the following known result. 
\begin{lemma}\label{Lem. sum}
	Let $k$ be an integer. Then 
	$$\sum_{x\in\mathbb{F}_q\setminus\{0\}}x^k=\begin{cases}
		-1 & \mbox{if}\ p-1\mid k,\\
		0  & \mbox{otherwise}.
	\end{cases}$$
\end{lemma}

Now we are in a position to prove our main results. For simplicity, the summations $\sum_{x\in\mathbb{F}_q}$ and $\sum_{x\in\mathbb{F}_q\setminus\{0\}}$ are abbreviated as $\sum_{x}$ and $\sum_{x\neq 0}$ respectively. 

{\noindent\bf Proof of Theorem \ref{Thm. A}.} Fix a primitive element $g$ of $\mathbb{F}_q$. Then one can verify that 
\begin{align*}
	\det S_q&=\prod_{i=1}^{q-1}a_i^2\cdot\det\left[\left(\left(\frac{a_j}{a_i}\right)^2+\frac{a_j}{a_i}+1\right)\chi\left(\left(\frac{a_j}{a_i}\right)^2+\frac{a_j}{a_i}+1\right)\right]_{1\le i,j\le q-1}\\
	&=\det\left[\frac{1}{g^{j-i}}\left(g^{2(j-i)}+g^{j-i}+1\right)^{\frac{q+1}{2}}\right]_{0\le i,j\le q-2}.
\end{align*}
Let $t_i=g^{-i}(g^{2i}+g^i+1)^{\frac{q+1}{2}}$ for $0\le i\le q-2$. Then 
$$\det S_q=\det C(t_0,t_1,\cdots,t_{q-2})$$
and $t_i=t_{q-1-i}$ for $1\le i\le q-3$. Applying Lemma \ref{Lem. the key lemma} there is an element $u\in\mathbb{F}_q$ such that 
\begin{equation}\label{Eq.1 in the proof of Thm. A(1)}
	\det S_q=\left(\sum_{i=0}^{q-2}t_i\right)\left(\sum_{i=0}^{q-2}(-1)^it_i\right)u^2.
\end{equation}
We first evaluate $\sum_{i=0}^{q-2}t_i$.
\begin{align*}
	\sum_{i=0}^{q-2}t_i
	&=\sum_{x\neq0}\frac{1}{x}\left(x^2+x+1\right)^{\frac{q+1}{2}}\\
	&=\sum_{x\neq0}\left(x+\frac{1}{x}+1\right)\cdot\left(x^2+x+1\right)^{\frac{q-1}{2}}\\
	&=2\sum_{x\neq0}x\cdot(x^2+x+1)^{\frac{q-1}{2}}+\sum_{x\neq0}(x^2+x+1)^{\frac{q-1}{2}}\\
	&=-1+2\sum_{x}x\left(\left(x+\frac{1}{2}\right)^2+\frac{3}{4}\right)^{\frac{q-1}{2}}+\sum_{x}\left(\left(x+\frac{1}{2}\right)^2+\frac{3}{4}\right)^{\frac{q-1}{2}}\\
	&=-1+2\sum_{x}\left(x-\frac{1}{2}\right)\left(x^2+\frac{3}{4}\right)^{\frac{q-1}{2}}+\sum_{x}\left(x^2+\frac{3}{4}\right)^{\frac{q-1}{2}}=-1.
\end{align*}
Hence we obtain 
\begin{equation}\label{Eq.2 in the proof of Thm. A(1)}
	\sum_{i=0}^{q-2}t_i=-1.
\end{equation}
Next we turn to $\sum_{i=0}^{q-2}(-1)^it_i$.
\begin{align*}
	\sum_{i=0}^{q-2}(-1)^it_i
	&=\sum_{x\neq 0}\frac{1}{x}\cdot\left(x^2+x+1\right)\cdot\chi\left(\frac{1}{x}\right)\chi\left(x^2+x+1\right)\\
	&=\sum_{x\neq 0}\left(x+\frac{1}{x}+1\right)^{\frac{q+1}{2}}=-T_{\frac{q+1}{2}}.
\end{align*}
The last equality follows from Lemma \ref{Lem. sum}. We therefore obtain 
\begin{equation}\label{Eq.3 in the proof of Thm. A(1)}
	\sum_{i=0}^{q-2}(-1)^it_i=-T_{\frac{q+1}{2}}.
\end{equation}
Combining (\ref{Eq.2 in the proof of Thm. A(1)}) and (\ref{Eq.3 in the proof of Thm. A(1)}) with (\ref{Eq.1 in the proof of Thm. A(1)}), we see that $\det S_q=T_{\frac{q+1}{2}}\cdot u^2$ for some $u\in\mathbb{F}_q$. 

This completes the proof.\qed

\section{Proof of Theorem \ref{Thm. B}}

We begin with the following known result (see \cite[Lemma 10]{K}).

\begin{lemma}\label{Lem. K}
	Let $R$ be a commutative ring and let $n$ be a positive integer. For any polynomial $P(T)=p_{n-1}T^{n-1}+\cdots+p_1T+p_0\in R[T]$ we have  
	$$\det\left[P(X_iY_j)\right]_{1\le i,j\le n}=\prod_{i=0}^{n-1}p_i\prod_{1\le i<j\le n}\left(X_j-X_i\right)\left(Y_j-Y_i\right).$$
\end{lemma}

We also need the following lemma.

\begin{lemma}\label{Lem. values of polynomial}
	Let $q$ be an odd prime. Then for any non-zero element $a\in\mathbb{F}_q$ we have $(a^2+a+1)^{\frac{q+1}{2}}=f(a)$, where 
	\begin{equation}\label{Eq. def of f}
	f(T)=\frac{11}{8}+T+\frac{11}{8}T^2+\sum_{k=-(q-5)/2}^{(q-5)/2}\binom{\frac{q+1}{2}}{k}_2T^{k+\frac{q+1}{2}}
	\end{equation}
is a polynomial over $\mathbb{F}_q$. 
\end{lemma}

\begin{proof} As $a\neq0$, we have $a^{q+k}=a^{k+1}$ for any integer $k$. Using this and $\binom{n}{k}_2=\binom{n}{-k}_2$ one can verify that $(a^2+a+1)^{\frac{q+1}{2}}$ is equal to
\begin{align*}
&\sum_{k=-(q-5)/2}^{(q-5)/2}\binom{\frac{q+1}{2}}{k}_2a^{k+\frac{q+1}{2}}
+\left(\binom{\frac{q+1}{2}}{\frac{q+1}{2}}_2+\binom{\frac{q+1}{2}}{\frac{q-3}{2}}_2\right)(1+a^2)+\left(\binom{\frac{q+1}{2}}{\frac{q-1}{2}}_2+\binom{\frac{q+1}{2}}{\frac{q-3}{2}}_2\right)a\\
=&\frac{11}{8}+a+\frac{11}{8}a^2+\sum_{k=-(q-5)/2}^{(q-5)/2}\binom{\frac{q+1}{2}}{k}_2a^{k+\frac{q+1}{2}}.
\end{align*}
The last equality follows from (below the trinomial coefficient $\binom{n}{k}_2$ is viewed as an element of $\mathbb{F}_q$).
$$\binom{\frac{q+1}{2}}{\frac{q+1}{2}}_2=1,\  \binom{\frac{q+1}{2}}{\frac{q-1}{2}}_2=\frac{q+1}{2}=\frac{1}{2},\ \binom{\frac{q+1}{2}}{\frac{q-3}{2}}_2=\frac{1}{2}\cdot\frac{q+1}{2}\cdot\frac{q+3}{2}=\frac{3}{8}.$$

This completes the proof.
\end{proof}
Now we are in a position to prove our last result.

{\noindent\bf Proof of Theorem \ref{Thm. B}.} By Lemma \ref{Lem. values of polynomial} one can verify that 
\begin{align*}
\det S_q&=\prod_{i=1}^{q-1}a_i^{q+1}\cdot\det\left[\left(\left(\frac{a_j}{a_i}\right)^2+\frac{a_j}{a_i}+1\right)^{\frac{q+1}{2}}\right]_{1\le i,j\le q-1}\\
&=\det\left[\left(\left(\frac{a_j}{a_i}\right)^2+\frac{a_j}{a_i}+1\right)^{\frac{q+1}{2}}\right]_{1\le i,j\le q-1}\\
&=\det\left[f\left(\frac{a_j}{a_i}\right)\right]_{1\le i,j\le q-1},
\end{align*}
where $f$ is defined by (\ref{Eq. def of f}).

Now applying Lemma \ref{Lem. K} we obtain that $\det S_q$ is equal to 
\begin{equation}\label{Eq.1 in the proof of Thm. B}
\det\left[f\left(\frac{a_j}{a_i}\right)\right]_{1\le i,j\le q-1}=\frac{121}{64}\cdot \binom{\frac{q+1}{2}}{0}_2\cdot\prod_{k=1}^{(p-5)/2}\binom{\frac{q+1}{2}}{k}_2^2\cdot\prod_{1\le i<j\le q-1}\left(a_j-a_i\right)\left(\frac{1}{a_j}-\frac{1}{a_i}\right).	
\end{equation}
By \cite[Eq. (3.3)]{WSN} we further have 
\begin{equation}\label{Eq.2 in the proof of Thm. B}
	\prod_{1\le i<j\le q-1}\left(a_j-a_i\right)\left(\frac{1}{a_j}-\frac{1}{a_i}\right)=1.
\end{equation}
Hence by (\ref{Eq.1 in the proof of Thm. B}) and (\ref{Eq.2 in the proof of Thm. B}) we obtain 
\begin{equation}\label{Eq.3 in the proof of Thm. B}
	\det S_q=\frac{121}{64}\cdot \binom{\frac{q+1}{2}}{0}_2\cdot\prod_{k=1}^{(p-5)/2}\binom{\frac{q+1}{2}}{k}_2^2.
\end{equation}
As $q>5$ and $(q,22)=1$, by (\ref{Eq.3 in the proof of Thm. B}) we see that 
$$\det S_q= 0\Leftrightarrow \binom{\frac{q+1}{2}}{k}_2\equiv 0\pmod p\ \text{for some}\ 0\le k\le (p-5)/2,$$
where $p$ is the characteristic of $\mathbb{F}_q$. 
This completes the proof.\qed

\end{document}